\documentclass[10pt]{amsart}
\usepackage[colorlinks=true,
pdfstartview=FitV, linkcolor=blue, citecolor=green,
urlcolor=]{hyperref}
\usepackage{amsmath,amsfonts,latexsym,amssymb}
\usepackage{amssymb,amsfonts, amsmath}
\usepackage{mathrsfs}
\usepackage[latin1]{inputenc}
\usepackage[T1]{fontenc}
\usepackage{ae,aecompl}
\usepackage{braket}
\usepackage{comment}
\usepackage{amssymb,amsfonts, amsmath}
 \usepackage{mathtools} 
\usepackage[all,arc]{xy}
\usepackage{enumerate}
\usepackage{mathrsfs}
\usepackage{color}
\usepackage{subfig}
\usepackage{verbatim} 
\usepackage{amssymb,amsfonts, amsmath}
\usepackage[all,arc]{xy}
\usepackage{enumerate}
\usepackage{mathrsfs}
\usepackage{hyperref}
\usepackage{xstring}
\usepackage{amsmath}
\usepackage[makeroom]{cancel}
\usepackage{mathabx}

\newtheorem{theorem}{Theorem}[section]
\newtheorem{lemma}[theorem]{Lemma}
\newtheorem{proposition}[theorem]{Proposition}

\renewcommand{\geq}{\geqslant}
\newcommand{\equaldef}{\overset{\mathrm{def}}{=}}

\long\def\@savemarbox#1#2{\global\setbox#1\vtop{\hsize\marginparwidth 
  \@parboxrestore\tiny\raggedright #2}}
\newcommand{\N}{\mathbb N}

\newcommand{\GL}{\mathsf{GL}}
\newcommand{\Sp}{\mathsf{Sp}}

\usepackage{braket}
\usepackage{amsmath}
\makeatletter
\renewcommand*\env@matrix[1][\arraystretch]{%
\edef\arraystretch{#1}%
\hskip -\arraycolsep
\let\@ifnextchar\new@ifnextchar
\array{*\c@MaxMatrixCols c}}
\makeatother
\date{\today}

\title{Residually finite non linear hyperbolic groups}

\date{\today}
\author[Tholozan]{Nicolas Tholozan}
\author[Tsouvalas]{Konstantinos Tsouvalas}
\AtEndDocument{\bigskip{\footnotesize%
  \textsc{D\'epartement de Math\'ematiques et Applications, ENS-PSL, 45 rue d'Ulm,
75230 Paris Cedex 5
France} \par  
 \textit{E-mail address}: \texttt{nicolas.tholozan@ens.fr} 
\vspace{0.2cm}

\textsc{CNRS, Laboratoire Alexander Grothendieck, Institut des Hautes \'Etudes Scientifiques, 
Universite Paris-Saclay, 35 route de Chartres, 91440 Bures-sur-Yvette, France} \par  
 \textit{E-mail address}: \texttt{tsouvkon@ihes.fr}
\textsc{}}}

\begin{document}
\frenchspacing
\maketitle

\begin{abstract} We exhibit the first examples of residually finite non-linear Gromov hyperbolic groups. Our examples are constructed as amalgamated products of torsion-free cocompact lattices in the rank 1 Lie group $\mathsf{Sp}(d,1)$, $d \geq2$, along maximal cyclic subgroups.
\end{abstract}
\section{Introduction}

Recall that a group $G$ is called {\em residually finite} if for every element \hbox{$g\in G\smallsetminus \{1\}$} there exists a finite group $F$ and a group homomorphism $\phi:G \rightarrow F$ with $\phi(g)\neq 1$. A long standing open question of Gromov \cite{Gromov} asks whether every hyperbolic group is residually finite. Since every finitely generated linear group is residually finite by Malcev's theorem, a negative answer to Gromov's question should be considered among non-linear hyperbolic groups. Using the superrigidity theorems of Corlette \cite{Cor} and Gromov--Schoen \cite{GS}, M. Kapovich constructed in \cite{Kap} the first examples of non-linear hyperbolic groups as quotients of cocompact lattices in the simple rank 1 Lie group $\mathsf{Sp}(d,1)$, $d\geq 2$. Other examples of non-linear hyperbolic groups were constructed by Canary, Stover and the second author in \cite{CST} and by the authors of this note in \cite{TT}, as amalgamated products or HNN extensions of quaternionic supperigid lattices along infinite cyclic or quasiconvex free \hbox{subgroups of rank at least two.}

\par In the light of the previous discussion, it is natural to ask which of these classes of non-linear hyperbolic groups can be shown to be residually finite. Unfortunately, this question is difficult in most cases, and deeply connected to other well-known problems in geometric group theory. For instance, residual finiteness of Kapovich's examples is related to the congruence subgroup conjecture for quaternionic lattices, see \cite[\S4]{Lubotzky}. Similarly, residual finiteness of amalgamated products along quasi-convex malnormal subgroups is related to the seperability of these subgroups in the ambient lattice, see \cite[\S2]{LN} and \cite{AGM}. In view of these difficulties, most of these non-linear examples actually provide good candidates for non-residually finite hyperbolic groups.

Nonetheless, in this note we show that one of our constructions in \cite{TT} provides hyperbolic non-linear residually finite groups. They seem to be the first examples of such groups.\footnote{See Mark Sapir's answer to this MathOverflow question : \\ \url{https://mathoverflow.net/questions/396895/examples-of-nonlinear-residually-finite-hyperbolic-groups}.} These are constructed as amalgamated products of cocompact lattices in $\mathsf{Sp}(d,1)$, $d \geq 2$, along maximal cyclic subgroups, hence they are also $\textup{CAT}(0)$ groups.

\begin{theorem} \label{mainthm} Let $\Gamma_1$ and $\Gamma_2$ be two cocompact lattices in $\mathsf{Sp}(d,1)$, $d \geq 2$. There exist finite-index subgroups $\Gamma'_1$ of $\Gamma_1$ and $\Gamma_2'$ of $\Gamma_2$ such that for every non-trivial primitive\footnote{An element $g\in \Gamma$ is called {\em primitive} if whenever $g=h^m$ for some $m\in \mathbb{N}$ then $m=1$.} elements $\gamma_1 \in \Gamma'_1$ and $\gamma_2\in \Gamma_2'$ with different translation lengths in the symmetric space of $\mathsf{Sp}(d,1)$, the amalgamated product \[\Gamma_{1}' \ast_{\gamma_1=\gamma_2} \Gamma_{2}'\] is a residually finite non-linear Gromov hyperbolic group.\end{theorem}

Note that there are several known examples of finitely generated residually finite non-linear groups. These include one-relator ascending HNN extensions of free groups exhibited by Drutu--Sapir in \cite{DS}, the automorphism groups $\textup{Aut}(F_n)$, $n \geq 3$, (see \cite{FP}) and residually finite groups containing infinite torsion $p$ and $q$-subgroups for two distinct primes $p,q$ (see e.g. \cite[Rmk. 3.4]{Nica}). More recently, Chong--Wise in \cite{CW} constructed an uncountable family of finitely generated residually finite groups. Most of them are not linear, since there are only countably many finitely generated linear groups, as Sami Douba pointed out to us. However, none of these examples are hyperbolic.

Let us now provide some details on our construction. The non-linearity of the amalgamated product $\Delta(\gamma_1,\gamma_2)\equaldef \Gamma_1\ast_{\gamma_1=\gamma_2} \Gamma_2$ is proved in \cite{TT} following the point of view of the constructions in \cite{CST}, and relies on the supperigidity theorems of Corlette \cite{Cor} and Gromov--Schoen \cite{GS}. We sketch the proof in Section \ref{s:NonLinear} for completeness.

%

In order to prove the residual finiteness of $\Delta(\gamma_1,\gamma_2)$, one needs to construct finite quotients of $\Gamma_1$ and $\Gamma_2$ in which $\gamma_1$ and $\gamma_2$ have the same arbitrarily large order. To guarantee this property, we appeal to a theorem of Platonov \cite{Platonov} saying that every finitely generated linear group admits an abundance of descending sequences of normal finite-index subgroups, where quotients between successive terms are $p$-groups for some primes $p\in \mathbb{N}$. We deduce the following residual finiteness result:

\begin{theorem}\label{mainthm2} Let $\Gamma_1$ and $\Gamma_2$ be non-elementary Gromov hyperbolic groups which are linear over characteristic zero. There exist finite-index subgroups $\Gamma_1'$ of $\Gamma_1$ and $\Gamma_2'$ of $\Gamma_2$ with the property that for every $\gamma_1\in \Gamma'_1$ and $\gamma_2\in \Gamma_2'$ primitive elements, the amalgamated product $$\Gamma'_1 \ast_{\gamma_1=\gamma_2}\Gamma'_2$$ is a residually finite Gromov hyperbolic group.\end{theorem}

\medskip

\noindent {\bf Acknowledgements.}

This project received funding from the European Research Council (ERC) under the European's Union Horizon 2020 research and innovation programme (ERC starting grant DiGGeS, grant agreement No 715982). We would like to thank Francois Dahmani, Sami Douba and Alan Reid for interesting discussions. We are also very grateful to Ga\"etan Chenevier for sharing interesting arguments about orders of elements in finite quotients of arithmetic lattices, which go beyond the goal of this paper.

\section{Non linear amalgamated products of superrigid lattices}  \label{s:NonLinear}

We recall here the construction of non-linear hyperbolic groups from \cite{TT} that will answer Theorem \ref{mainthm}.

\begin{theorem}[Theorem 1.3 of \cite{TT}]\label{nonlinear-amalgam} Let $\Gamma_1$ and $\Gamma_2$ be two lattices in $\mathsf{Sp}(d,1)$, $d\geq 2$. Assume that $\gamma_1\in \Gamma_1$ and $\gamma_2\in \Gamma_2$ are two infinite order elements with different translation lengths in the symmetric space of $\mathsf{Sp}(d,1)$. Then for every field $k$ and $n \in \mathbb{N}$, every representation $\rho:\Gamma_1\ast_{\gamma_1=\gamma_2}\Gamma_2 \rightarrow \GL_r(k)$ \hbox{maps $\Gamma_1$ and $\Gamma_2$ to a finite group.} In particular, the amalgamated product $\Gamma_1\ast_{\gamma_1=\gamma_2}\Gamma_2$ is not linear.\end{theorem}

Let us sketch the proof here for completeness. The details are in \cite[\S5]{TT}. The cornerstone of the proof is the superrigidity theorem of Corlette \cite{Cor} and Gromov--Schoen \cite{GS}, which say that a linear representation of a $\Sp(d,1)$ lattice with infinite image essentially extends to $\Sp(d,1)$. The other ingredient comes from the representation theory of $\mathsf{Sp}(d,1)$: there is a constant $c$ such that, for every proximal continuous homomorphism $\psi:\mathsf{Sp}(d,1)\rightarrow \mathsf{GL}_r(\mathbb{C})$ and every $g\in \Sp(d,1)$, we have
\[\log\frac{\lambda_1\big(\psi(g))}{\lambda_2(\psi(g))}=c \cdot \ell_{\mathbb{H}{\bf H}^d}(g)~,\]
where $\lambda_1(\psi(g))\geq \lambda_2(\psi(g))\geq \cdots \geq \lambda_r(\psi(g))$ are the moduli of the eigenvalues of $\psi(g)$ and $\ell_{\mathbb{H}{\bf H}^d}(g)$ denotes the translation length of $g\in \mathsf{Sp}(d,1)$ acting on the quaternionic hyperbolic space $\mathbb{H}{\bf H}^d$. In the context of Theorem \ref{nonlinear-amalgam}, the superrigidity of $\Gamma_1$ and $\Gamma_2$ and the fact that $\gamma_1,\gamma_2\in \mathsf{Sp}(d,1)$ have distinct translation lengths implies that two linear representations of $\rho_i$ of $\Gamma_i$ with infinite image cannot satisfy
\[\rho_1(\gamma_1) = \rho_2(\gamma_2)~.\]
Thus the amalgamated product does not admit a faithful linear representation.

\section{Residual finiteness of amalgamated products} 

There are certain ways of proving the residual finiteness of amalgamated products of residually finite groups, as soon as compatibility conditions hold for the amalgamated subgroups, see for example \cite[Prop. 1]{Baumslag}. We give here the following refinement of Baumslag's proposition for Gromov hyperbolic groups.

\begin{lemma} \label{RF} Let $\Gamma_1$ and $\Gamma_2$ be two torsion-free Gromov hyperbolic groups and let $\gamma_1\in \Gamma_1$ and $\gamma_2\in \Gamma_2$ be two primitive elements. Suppose that there exist decreasing sequences $\big (\Gamma_{1,n}\big)_{n=1}^{\infty}$ and $\big (\Gamma_{2,n} \big)_{n=1}^{\infty}$ of finite-index normal subgroups of $\Gamma_1$ and $\Gamma_2$ respectively with the following properties:

\begin{itemize}
\item[(i)] $\bigcap_{n=1}^{\infty}\Gamma_{i,n}=\{1\}$ for $i \in \{1,2\}$,

\item[(ii)] $[\langle \gamma_1\rangle : \langle \gamma_1 \rangle \cap \Gamma_{1,n}]= [\langle \gamma_2\rangle : \langle \gamma_2 \rangle \cap \Gamma_{2,n}]$ for every $n \in \mathbb{N}$.
\end{itemize}

Then the amalgamated product $\Gamma_1 \ast_{\gamma_1 =\gamma_2}\Gamma_2$ is residually finite.\end{lemma}

Before we proceed with the proof of this lemma, let us observe that maximal cyclic subgroups of residually finite torsion-free hyperbolic groups are separable.

\begin{proposition} \label{sep} Let $\Gamma$ be a torsion-free Gromov hyperbolic group. Suppose that $\big(\Gamma_n\big )_{n=1}^{\infty}$ is a decreasing sequence of finite-index normal subgroups of $\Gamma$ with $\bigcap_{n=1}^{\infty}\Gamma_n=\{1\}$. Let $\gamma\in \Gamma$ be a primitive element. Then $$\bigcap_{n=1}^{\infty}\langle \gamma \rangle \Gamma_n=\langle \gamma \rangle.$$\end{proposition}

\begin{proof} Let us set $H\equaldef \bigcap_{n=1}^{\infty}\langle \gamma \rangle \Gamma_n$. Let $h \in H$, fix $n \in \mathbb{N}$ and write $h=\eta_n \gamma^{s_n}$ for some $s_n\in \mathbb{N}$ and $\gamma_n \in \Gamma_n$. Observe that $[h,\gamma]=\eta_n \gamma \eta_{n}^{-1}\gamma^{-1} \in \Gamma_n$ since $\Gamma_n$ is a normal subgroup of $\Gamma$. We thus have $[h,\gamma]\in \bigcap_{n=1}^\infty \Gamma_n = \{1\}$. Hence $H$ centralizes the cyclic group $\langle \gamma \rangle$. 
Since $\Gamma$ is torsion-free hyperbolic and $\langle \gamma \rangle$ is maximal cyclic, it is equal to its centralizer \cite{Gromov}. Hence $H = \langle \gamma \rangle$.\end{proof}

In the light of Proposition \ref{sep}, the proof of Lemma \ref{RF} is quite standard. We provide a proof for the reader's convenience.
\medskip

\begin{proof}[Proof of Lemma \ref{RF}] Let us set $\Delta = \Gamma_1 \ast_{\gamma_1=\gamma_2}\Gamma_2$. For every $g \in \Delta$, we shall exhibit a finite-index normal subgroup $N$ of $\Delta$ with $g\in \Delta \smallsetminus N$. 
\par First, let us assume that $g$ does not lie in a conjugate of $\Gamma_1$ or $\Gamma_2$. Up to conjugation, we may write $g=w_1\eta_1\cdots w_s \eta_s$ where $w_1,\ldots, w_s \in \Gamma_1 \smallsetminus \langle \gamma_1 \rangle$ and $\eta_1,\ldots, \eta_s \in \Gamma_2 \smallsetminus \langle \gamma_2 \rangle$. 

By assumption (i) and Proposition \ref{sep} there exists $m \in \mathbb{N}$ large enough such that $w_1,\ldots w_s \in \Gamma_1 \smallsetminus \langle \gamma_1 \rangle \Gamma_{1,m}$ and $\eta_1,\ldots \eta_s \in \Gamma_2 \smallsetminus \langle \gamma_2 \rangle \Gamma_{2,m}$. By assumption (ii) we have $[\langle \gamma_1 \rangle: \langle \gamma_1 \rangle \cap \Gamma_{1,m}]=[\langle \gamma_2 \rangle: \langle \gamma_2 \rangle \cap \Gamma_{2,m}]$, so $\langle \bar \gamma_1  \rangle$ and $\langle \bar \gamma_2\rangle$ are finite cyclic groups of $\Gamma_1/\Gamma_{1,m}$ and $\Gamma_2/\Gamma_{2,m}$ of the same order. Thus there exists a surjective homomorphism 
\[\pi:\Delta \rightarrow \Gamma_1/\Gamma_{1,m}\ast_{\bar \gamma_1 = \bar \gamma_2} \Gamma_2/\Gamma_{2,m}\]
restricting to the quotient morphism $\Gamma_i \to \Gamma_i/\Gamma_{i,m}$ for $i\in \{1,2\}$. By construction, $\pi(w_1),\ldots, \pi(w_s) \in \Gamma_1/\Gamma_{1,m} \smallsetminus \langle \bar \gamma_1 \rangle$ and $\pi(\eta_1),\ldots, \pi(\eta_s) \in \Gamma_2/\Gamma_{2,m} \smallsetminus \langle \bar \gamma_2\rangle$, and we conclude that \[\pi(g)=\pi(w_1)\pi(\eta_1)\cdots \pi(w_s)\pi(\eta_s)\neq 1~.\] Now, an amalgamated product of finite groups is virtually free (see for instance \cite{Serre}) hence residually finite, so there exists a finite group $F$ and a surjective group homomorphism 
\[\varphi: \Gamma_1/\Gamma_{1,m}\ast_{\bar \gamma_1 = \bar \gamma_2 } \Gamma_2/\Gamma_{2,m} \rightarrow F\]
 with $\varphi(\pi(g))\neq 1$. In particular, $g \in \Delta \smallsetminus \textup{ker}(\pi \circ \varphi)$.
 
In the case where $g=whw^{-1}$ for some $w\in \Delta$ and $h\in \Gamma_1$ (resp. $h\in \Gamma_2$) we choose $n\in \mathbb{N}$ large enough such that $h \in \Gamma_1 \smallsetminus \Gamma_{1,n}$ (resp. $h \in \Gamma_2 \smallsetminus \Gamma_{2,n}$). We obtain a surjective group homomorphism $\pi:\Delta \rightarrow \Gamma_1/\Gamma_{1,n} \ast_{\bar \gamma_1 = \bar \gamma_2}\Gamma_2/\Gamma_{2,n}$ with $\pi(h)\neq 1$, hence $\pi(g) \neq 1$. Again, since $\Gamma_1/\Gamma_{1,n} \ast_{\bar \gamma_1 = \bar \gamma_2}\Gamma_2/\Gamma_{2,n}$ is residually finite, $g$ survives in a finite quotient of $\Delta$. 

\end{proof}

\section{Platonov's theorem}

By Lemma \ref{RF}, in order to prove residual finiteness of our amalgamated hyperbolic groups, we need to construct sufficiently many quotients of these groups in which the amalgamated cyclic subgroups have the same order. These will be given by the following theorem of Platonov \cite{Platonov} which shows that linear finitely generated groups are residually $p$-finite for some $p$ (i.e. every non-trivial element survives in a finite quotient which is a $p$-group).

\begin{theorem}\textup{(Platonov \cite{Platonov})} \label{Platonov} Let $k$ be a field of characteristic zero and $\Gamma$ be a finitely generated subgroup of $\GL_r(k)$. Then, for all but finitely many primes $p \in \mathbb{N}$, there exists a decreasing sequence of finite-index normal subgroups $\big(\Gamma(p^n)\big)_{n=1}^{\infty}$ of $\Gamma$ with the following properties:

\begin{itemize}
\item[(i)] $\bigcap_{n=1}^{\infty}\Gamma(p^n)=\{1\}.$

\item[(ii)] for $n\in \mathbb{N}$, every non-trivial element of $\Gamma(p^n)/\Gamma(p^{n+1})$ has order equal to $p$.
\end{itemize}

In particular, $\Gamma$ is virtually residually $p$-finite for all but finitely many primes $p$. \end{theorem}

The sequence $\big(\Gamma(p^n)\big)_{n=1}^{\infty}$ is constructed in the following way: let $A$ be the domain generated by the matrix entries of elements of (a finite generating set of) $\Gamma$ and let $I\subset A$ be a maximal ideal such that $A/I$ is a finite field of characteristic $p$. The finite index normal subgroup $\Gamma(p^n)$ is then defined as the kernel of the morphism 
\[\Gamma \subset \GL_r(A) \to \GL_r(A/I^n)~.\]
A detailed proof of Platonov's theorem is given in \cite[Thm. 3.1]{Nica}.

\section{Proof of the theorems}

We now have all the tools to conclude the proof of Theorem \ref{mainthm2} and Theorem \ref{mainthm}.

\begin{proof}[Proof of Theorem \ref{mainthm2}] 

Note first that, by Malcev's theorem, $\Gamma_1$ and $\Gamma_2$ are residually finite, and by Selberg's lemma, up to passing to finite-index subgroups, we may assume that $\Gamma_1$ and $\Gamma_2$ are torsion-free. By Theorem \ref{Platonov} there exists a prime $p\in \mathbb{N}$ and descending sequences $\big\{\Gamma_1(p^n)\big\}_{n=1}^{\infty}$ and $\big\{\Gamma_2(p^n)\big\}_{n=1}^{\infty}$ of finite-index normal subgroups of $\Gamma_1$ and $\Gamma_2$ respectively such that:

\begin{itemize}
\item[(i)] $\bigcap_{n=1}^{\infty}\Gamma_1(p^n)=\{1\}$ and $\bigcap_{n=1}^{\infty}\Gamma_2(p^n)=\{1\}$,

\item[(ii)]
for $n\geq 1$ and $i\in \{1,2\}$, every non-trivial element of $\Gamma_{i}(p^n)/\Gamma_{i}(p^{n+1})$ \hbox{has order equal to $p$.}
\end{itemize}

Let us set $\Gamma_1' = \Gamma_1(p)$ and $\Gamma_2'= \Gamma_2(p)$. Let $\gamma_1\in \Gamma_1'$ and $\gamma_2\in \Gamma_2'$ be two non-trivial primitive elements. We claim that the amalagamated product \[\Delta(\gamma_1,\gamma_2)\equaldef \Gamma_1'\ast_{\gamma_1=\gamma_2}\Gamma_2'\] is a residually finite hyperbolic group.

Since $\Gamma_i$ is torsion-free for $i\in \{1,2\}$, $\langle \gamma_i\rangle$ is malnormal in $\Gamma_i(p)$ and the hyperbolicity of $\Delta(\gamma_1,\gamma_2)$ follows immediately by the Bestvina--Feighn combination theorem \cite{BF}. 

Observe that for every $n\in \mathbb{N}$ and $i\in \{1,2\}$ the order of $\bar \gamma_{i}$ in $\Gamma_i/\Gamma_i(p^n)$ is a power of $p$ since $\gamma_i\in \Gamma_i(p)$ and $\Gamma_i(p)/\Gamma_i(p^n)$ is a finite $p$-group. For every $n\in \mathbb{N}$ and $i\in \{1,2\}$, define $a_i(n)$ as the integer such that $\bar \gamma_i$ has order $p^{a_i(n)}$ in $\Gamma_i/\Gamma_i(p^n)$. Since $\Gamma_i(p^n)$ is a decreasing sequence of normal subgroups with $\bigcap_{n=1}^{\infty}\Gamma_i(p^n)=\{1\}$, the sequence $a_i(n)$ is increasing and unbounded. Moreover, by Property (ii) of Platonov's theorem, $\gamma_i^{p^{a_i(n)}} \in \Gamma_i(p^n)$ has order $1$ or $p$ in $\Gamma_i(p^n)/\Gamma_i(p^{n+1})$, which implies that $a_i(n+1)-a_i(n) \in \{0,1\}$ for $i\in \{1,2\}$. We conclude that $a_i: \N \to \N$ is surjective for $i \in \{1,2\}$.

Let $(k_n)_{n=1}^{\infty}$ and $(l_n)_{n=1}^{\infty}$ be increasing sequences such that \[a_1(k_n)=a_2(l_n)=n~.\]
Finally, set $\Gamma'_{1,n} = \Gamma_1(p^{k_n})$ and $\Gamma'_{2,n} = \Gamma_2(p^{l_n})$. By the definition of $k_n$ and $l_n$, the order of $\bar \gamma_i$  in $\Gamma_i'/\Gamma'_{i,n}$ is $p^n$ for $i\in \{1,2\}$. Since $(k_n)$ and $(l_n)$ are unbounded, we have $\bigcap_{n=1}^{\infty} \Gamma_{i,n} = \{1\}$. The decreasing sequences of subgroups $\big (\Gamma'_{i,n}\big )_{n=1}^{\infty}$ satisfy hypotheses (i) and (ii) of Lemma~\ref{RF} and we conclude that $\Delta(\gamma_1,\gamma_2)$ is residually finite. \end{proof}

Now Theorem \ref{mainthm} follows straightforwardly from Theorem \ref{mainthm2} and Theorem \ref{nonlinear-amalgam}.

\begin{proof}[Proof of Theorem \ref{mainthm}] Let $\Gamma_1$ and $\Gamma_2$ be two cocompact lattices of $\mathsf{Sp}(d,1)$, $d\geq 2$. Since $\Gamma_1$ and $\Gamma_2$ are linear over $\mathbb{R}$, by Theorem \ref{mainthm2} there exist finite-index subgroups $\Gamma'_1$ and $\Gamma_2'$ of $\Gamma_1$ and $\Gamma_2$ respectively such that for every $\gamma_1\in \Gamma'_1$ and $\gamma_2\in \Gamma_2'$ primitive elements, the group $\Delta(\gamma_1,\gamma_2)$ is a residually finite Gromov hyperbolic group. On the other hand, when the translation lengths of $\gamma_1$ and $\gamma_2 \in \mathsf{Sp}(d,1)$ are different, the group $\Delta(\gamma_1,\gamma_2)$ is non-linear by Theorem \ref{nonlinear-amalgam}.

\end{proof}

\end{document}